\documentclass[a4paper,12pt,twoside]{amsart}
\usepackage{verbatim}
\usepackage[british]{babel}
\usepackage{amsmath}
\usepackage{amsthm}
\usepackage{amssymb}
\usepackage[nice]{nicefrac}

\theoremstyle{plain} 
\newtheorem{thm}{Theorem}[section]
\newtheorem{prop}[thm]{Proposition}
\newtheorem{cor}[thm]{Corollary}
\newtheorem{lemma}[thm]{Lemma}

\theoremstyle{definition}
\newtheorem*{acknow}{Acknowledgement}

\numberwithin{equation}{section}

\newcommand{\ol}[1]{\overline{#1}}

\renewcommand{\epsilon}{\varepsilon}
\renewcommand{\phi}{\varphi}

\begin{document}

\title{The Magnus Property for direct products}

\author{Benjamin Klopsch} \address{Mathematisches Institut,
  Heinrich-Heine-Universit\"at D\"usseldorf, 40225 D\"usseldorf, Germany}
\email{klopsch@math.uni-duesseldorf.de}

\author{Benno Kuckuck} \address{Mathematisches Institut,
  Heinrich-Heine-Universit\"at D\"usseldorf, 40225 D\"usseldorf, Germany}
\email{Benno.Kuckuck@uni-duesseldorf.de}

\keywords{Magnus property, direct products, torsion-free groups,
  crystallographic groups, Hantzsche-Wendt group.}

\subjclass[2010]{Primary 20E45; Secondary 20E26, 20H15}

\maketitle

\begin{abstract}
  A group $G$ is said to have the Magnus property if the following
  holds: whenever two elements $x,y$ have the same normal
  closure, then $x$ is conjugate to $y$ or to $y^{-1}$.  We prove: Let
  $p$ be an odd prime, and let $G,H$ be residually finite-$p$ groups
  with the Magnus property.  Then the direct product $G \times H$ has
  the Magnus property.

  By considering suitable crystallographic groups, we give an explicit
  example of finitely generated, torsion-free, residually finite groups
  $G,H$ with the Magnus property such that the direct product
  $G \times H$ does not have the Magnus property.
\end{abstract}


\section{Introduction}

A group $G$ is said to satisfy the Magnus property if the following
holds: whenever $x,y \in G$ have the same normal closure
$\langle x \rangle^G = \langle y \rangle^G$ then $x$ is conjugate to
$y$ or to $y^{-1}$ in~$G$.  A classical theorem of Magnus~\cite{Ma30}
shows that free groups possess the Magnus property.  More recently,
Bogopolski~\cite{Bo05}, Bogopolski and Sviridov~\cite{BoSv08} and
Feldkamp~\cite{Fe} established or disproved the Magnus property for
certain one-relator groups, using group-theoretic techniques.  In
particular, the fundamental groups of closed surfaces possess the
Magnus property.  For orientable surface groups of genus at least $2$
and non-orientable surface groups of genus at least $4$, this result
can also be obtained as a consequence of the fact that the Magnus
property holds in any elementary free group, that is in any group
having the same first-order theory as non-abelian free groups.

For a more systematic study of the Magnus property it is of interest
to find out under what conditions the property is preserved under
standard group constructions, such as free or direct products.
In~\cite{BoSv08} it is suggested that the Magnus property is
unconditionally preserved under direct products.  At a rather basic
level, finite cyclic groups show that this is in fact not the case;
see Corollary~\ref{cor:not-preserved-cyclic}.  Nevertheless, with
comparatively little effort we establish the following theorem.

\begin{thm} \label{thm:main}
  Let $p$ be an odd prime, and let $G, H$ be residually
  finite-$p$ groups.  If $G$ and $H$ have the Magnus property then the
  direct product $G \times H$ has the Magnus property.
\end{thm}

Denoting by
$\mathrm{Dr}_{i \in I} G_i = \{ (g_i) \in \prod_{i \in I} G_i \mid g_i
=1 \text{ for almost all } i \}$
the direct product of a family of groups $(G_i)_{i \in I}$, we record
as a sample the following consequence.
 
\begin{cor} \label{cor:direct-sum} Any direct product
  $\mathrm{Dr}_{i \in I} F_i$ of free groups has the Magnus property.
\end{cor}

Unsuccessful attempts to extend the above result to torsion-free
groups led us to the discovery of an instructive `counterexample',
illustrating that Theorem~\ref{thm:main} is in some sense already
close to optimal.  First we establish that the well-known group
\[
G = \langle x,y \mid x^{-1} y^2 x = y^{-2}, y^{-1} x^2 y = x^{-2}
\rangle \cong (C_2 \times C_2).\mathbb{Z}^3
\]
has the Magnus property; see Proposition~\ref{pro:G-group}.  This
group is sometimes called the Passman fours group or the
Hantzsche-Wendt group; it is a centreless torsion-free
$3$\nobreakdash-dimensional crystallographic group with holonomy group
$C_2 \times C_2$.

The group $G$ allows for various generalisations; e.g.,
see~\cite{HiSa86,GuSi99}.  We establish that a natural variant
$H \cong (C_3 \times C_3).\mathbb{Z}^8$, a centreless torsion-free
$8$-dimensional crystallographic group with holonomy group
$C_3 \times C_3$, also has the Magnus property; see
Proposition~\ref{pro:H-group}.  Finally, in
Proposition~\ref{pro:G-H-product}, we establish that $G \times H$
fails to have the Magnus property.  This leads to the following
consequence.

\begin{prop} \label{pro:not-preseverved-torsion-free}
  The Magnus property is not preserved under direct products within the
  class of finitely generated, torsion-free, residually finite groups.
\end{prop}


\section{Proofs of Theorem~\ref{thm:main} and Corollary~\ref{cor:direct-sum}}

The number of generators of a finite cyclic group $C_n$ of order $n$
is equal to $\phi(n)$, where $\phi$ denotes the Euler function.
Therefore $C_n$ has the Magnus property if and only if
$\phi(n) \leq 2$.  This yields the following proposition, and the
concrete example $C_{12}\cong C_4\times C_3$ gives rise to an
interesting corollary.

\begin{prop} For $n \in \mathbb{N}$, the cyclic group $C_n$ has the
  Magnus property if and only if $n \in\{1,2,3,4,6\}$.
\end{prop}

\begin{cor} \label{cor:not-preserved-cyclic}
  The Magnus property is not preserved under direct products.
\end{cor}

In the next section we will show that even
within the class of finitely generated, torsion-free, residually
finite groups the Magnus property is not preserved under direct
products.  Though this will require somewhat more effort,
our construction will be guided by the failure of the elementary example
above.

Now we concentrate on establishing conditions 
which ensure that the Magnus property holds for a direct product.

\begin{lemma} \label{lem:i-ii-iii}
  Let $G$ and $H$ be groups with the Magnus property.  Then
  $G \times H$ has the Magnus property if and only if, for all
  $(g,h) \in G \times H$, one of the following holds:
  \begin{enumerate}
  \item[(i)] $g$ is conjugate to $g^{-1}$ in $G$; 
  \item[(ii)] $h$ is conjugate to $h^{-1}$ in $H$;
  \item[(iii)]
    $\langle (g,h) \rangle^{G \times H} \neq \langle (g,h^{-1})
    \rangle^{G \times H}$.
  \end{enumerate}
\end{lemma}

\begin{proof}
  First suppose that $G \times H$ has the Magnus property, and let
  $(g,h) \in G \times H$ be such that (iii) fails, i.e.,
  $\langle (g,h) \rangle^{G \times H} = \langle (g,h^{-1}) \rangle^{G
    \times H}$.  Then $(g,h)$ is conjugate to $(g,h^{-1})$ or to
  $(g^{-1},h)$ in $G \times H$.  Consequently, at least one of (i), (ii) holds.

  Conversely, suppose that $G\times H$ does not have the Magnus
  property.  Hence there are elements $(g,h),(g',h') \in G\times H$
  with
  $\langle (g,h) \rangle^{G \times H} = \langle (g',h') \rangle^{G
    \times H}$,
  but which are not conjugate to each other or their respective
  inverses.  Projecting onto the first factor, we see that
  $\langle g \rangle^G = \langle g' \rangle^G$.  Since $G$ has the
  Magnus property, we may conjugate $(g',h')$ by an element of
  $G \times 1$ and pass to the inverse, if necessary, to reduce to the
  case that $g' = g$.  Projecting onto the second factor and using the
  Magnus property for $H$ in a similar way, we may assume that
  $(g',h') = (g,h^{-1})$.  Then $(g,h)$ is not conjugate to
  $(g,h^{-1})$ or $(g^{-1},h)$ in $G \times H$.  Consequently, (i),
  (ii) and (iii) fail simultaneously.
\end{proof}

\begin{lemma} \label{lem:Cm-Cn}
  Let $G = \langle x, y \mid x^m=y^n=[x,y]=1 \rangle \cong C_m \times
  C_n$ be the direct product of finite cyclic groups of orders $m,n$.
  Then $\langle xy \rangle = \langle xy^{-1} \rangle$ if and only if
  $\gcd(m,n) \in \{1,2\}$.
\end{lemma}

\begin{proof}
  Finite abelian groups are direct products of their
  Sylow $p$-subgroups.  Hence it suffices to prove the result for
  powers $m=p^k$ and $n=p^l$ of a prime~$p$.  For
  $\gcd(m,n) \in \{1,2\}$, we can now conclude $x = x^{-1}$ or
  $y = y^{-1}$, and hence
  $\langle xy \rangle = \langle xy^{-1} \rangle$.  Finally, suppose
  that $p \geq 3$ and $k,l \geq 1$, or respectively $p=2$ and
  $k,l \geq 2$.  Projecting onto $C_p \times C_p$, respectively
  $C_4 \times C_4$, we see that
  $\langle xy \rangle \neq \langle xy^{-1} \rangle$.
\end{proof}

\begin{prop} \label{pro:hinreichendes-Kriterium}
  Let $G$ and $H$ be groups with the Magnus property.  Then
  $G \times H$ has the Magnus property if, for all
  $(g,h) \in G \times H$ such that $g$ is not conjugate to $g^{-1}$ in $G$ 
 and $h$ is not conjugate to $h^{-1}$ in $H$, there exist $m,n \in
 \mathbb{N}$ with 
  \begin{enumerate}
  \item[(a)] $\gcd(m,n) \not \in \{1,2\}$;
  \item[(b)] the cyclic group $\langle g \rangle^G / \langle [g,x] \mid x \in G
    \rangle^G$ projects onto $C_m$;
  \item[(c)] the cyclic group $\langle h \rangle^H / \langle [h,y] \mid y \in H
    \rangle^H$ projects onto $C_n$.
  \end{enumerate}  
\end{prop}

\begin{proof}
  The result is a direct consequence of Lemmata~\ref{lem:i-ii-iii} and
  \ref{lem:Cm-Cn}.
\end{proof}

\begin{lemma} \label{lem:central-quotient}
  Let $G$ be a residually finite-$p$ group, and
  $x \in G\setminus\{1\}$.  Then there exists a finite $p$-group
  quotient $\ol G = G/N$ such that $\ol x = xN$ is non-trivial and
  central in~$\ol G$.
\end{lemma}

\begin{proof}
  Since $G$ is residually finite-$p$, we may assume that $G$ is a
  finite $p$-group.  Let $\gamma_k(G)$ be the smallest term in the
  lower central series of $G$ (i.e., $k$ maximal) such that
  $x \in \gamma_k(G)$, and take $N = \gamma_{k+1}(G)$.
\end{proof}

\begin{proof}[Proof of Theorem~\ref{thm:main}]
  Since $G$ and $H$ are residually finite-$p$ groups,
  Lemma~\ref{lem:central-quotient} implies that for any
  $(g,h) \in G \times H$ conditions (a), (b) and (c) in
  Proposition~\ref{pro:hinreichendes-Kriterium} are satisfied for
  $m=n=p$.  Thus $G \times H$ has the Magnus property.
\end{proof}

\begin{proof}[Proof of Corollary~\ref{cor:direct-sum}]
  Magnus' classical result~\cite{Ma30} on the Magnus property for free
  groups and an induction based on Theorem~\ref{thm:main} show that
  any direct product of finitely many free groups has the Magnus
  property.
  
  In an arbitrary direct product $G = \mathrm{Dr}_{i\in I} F_i$ any
  two elements $x,y \in G$ are contained in some normal subgroup
  $N \trianglelefteq G$ which is a direct product of a finite subset
  of the factors.  If $x$ and $y$ have the same normal closure,
  then, by the case for finitely many factors, $x$ is conjugate to $y$
  or $y^{-1}$, even within~$N$.
\end{proof}


\section{Crystallographic groups}

\subsection{} First we consider the group
\begin{equation} \label{equ:Passman-group} G = \langle x,y \mid
    x^{-1}y^2x=y^{-2}, y^{-1}x^2y=x^{-2} \rangle.
\end{equation}
Setting $z=xy$, one can easily check that the relations
\begin{align*}
  x^{-1}z^2x&=z^{-2} & z^{-1}x^2z&=x^{-2}\\
  y^{-1}z^2y&=z^{-2} & z^{-1}y^2z&=y^{-2}
\end{align*}
hold, so there are automorphisms sending $x$ and $y$
to any two out of $x$, $y$ and $z$.

Furthermore, the subgroup $A = \langle x^2,y^2,z^2 \rangle \leq G$ is free
abelian of rank $3$, with the indicated basis, has index
$\lvert G : A \rvert = 4$ and $G/A \cong C_2\times C_2$.  We note in
passing that $G/[G,G] \cong C_4 \times C_4$ and
\[
[x,y] = x^{-1} y^{-1} x y = x^{-1} y^{-1} z^2 y^{-1} x^{-1} = x^{-2}
y^2 z^2. 
\]
All this is consistent with the faithful matrix
representation of $G$ in $\mathrm{GL}_4(\mathbb{Q})$ given by 
\[
X =
\left[\begin{smallmatrix}
1 &  \nicefrac{-1}{2} & 0 & 0 \\
 & 1 & & \\
 & & -1 & \\
 & & & -1
\end{smallmatrix}\right]
\qquad \text{and} \qquad Y =
\left[\begin{smallmatrix}
1 & 0 & \nicefrac{1}{2} & \nicefrac{1}{2} \\
 & -1 & & \\
 & & 1 & \\
 & & & -1
\end{smallmatrix}\right].
\]
For instance, one calculates easily,
\[
X^2 =
\left[\begin{smallmatrix}
1 & -1 & 0 & 0\\
 & 1 & & \\
 & & 1 & \\
 & & & 1
\end{smallmatrix}\right],
\quad Y^2 =
\left[\begin{smallmatrix}
1 & 0 & 1 & 0\\
 & 1 & & \\
 & & 1 & \\
 & & & 1
\end{smallmatrix}\right],
\quad (XY)^2 =
\left[\begin{smallmatrix}
1 & 0 & 0 & 1\\
 & 1 & & \\
 & & 1 & \\
 & & & 1
\end{smallmatrix}\right],
\quad [X,Y] =
\left[\begin{smallmatrix}
1 & 1 & 1 & 1\\
 & 1 & & \\
 & & 1 & \\
 & & & 1
\end{smallmatrix}\right].
\] 
In the second part of the current section we explain how $G$ can
be regarded as a member $G_2$ of an infinte sequence of groups $G_p$, one
for each prime~$p$.

\begin{prop} \label{pro:G-group}
  The group $G$ defined in \eqref{equ:Passman-group}
  has the Magnus property.
\end{prop}

\begin{proof}
  We fix $g,h \in G$ with $\langle g \rangle^G = \langle h \rangle^G$
  and show that $h$ is conjugate to $g$ or to~$g^{-1}$.  In some of
  our calculations it helps to think of
  $A = \langle x^2 \rangle \times \langle y^2 \rangle \times \langle
  z^2 \rangle \cong \mathbb{Z}^3$
  as a direct sum of three $1$-dimensional modules for
  $\langle \ol x, \ol y \rangle \cong C_2 \times C_2$; compare the
  construction of the more general family of groups $G_p$ given below.  We
  distinguish two cases.

  \medskip

  \noindent \emph{Case 1}: $g \in A$.  In this case
  $g = x^{2\alpha} y^{2\beta} z^{2\gamma}$ with
  $\alpha, \beta, \gamma \in \mathbb{Z}$.  If $(\alpha,\beta,\gamma)$
  has at most one non-zero coordinate, $\beta = \gamma = 0$ say, then
  $\langle h \rangle^G = \langle g \rangle^G = \langle x^{2\alpha}
  \rangle$
  implies $h \in \{ x^{2\alpha}, x^{-2\alpha} \} = \{g,g^{-1}\}$.

  Now suppose that $(\alpha,\beta,\gamma)$ has precisely two non-zero
  coordinates, $\alpha, \beta \neq 0$ and $\gamma = 0$ say.  Then the
  relations $[g,x] = y^{-4\beta}$ and $[g,y] = x^{-4\alpha}$ show that
  \[
  \langle h \rangle^G = \langle g \rangle^G = \langle x^{2\alpha}
  y^{2\beta}, x^{4\alpha}, y^{4\beta} \rangle.
  \]
  This implies
  $h \in \{ x^{2\alpha} y^{2\beta}, x^{2\alpha} y^{-2\beta},
  x^{-2\alpha} y^{2\beta}, x^{-2\alpha} y^{-2\beta}\} =
  \{g,g^x,g^y,g^{-1}\}$.

  Finally, suppose that all coordinates of $(\alpha,\beta,\gamma)$ are
  non-zero. Then the relations
  \[
  [g,x] = y^{-4\beta} z^{-4\gamma}, \qquad [g,y] =
  x^{-4\alpha} z^{-4\gamma}, \qquad g^2 = x^{4\alpha} y^{4\beta}
  z^{4\gamma}
  \]
  show that
  \[
  \langle h \rangle^G = \langle g \rangle^G = \langle x^{2\alpha}
  y^{2\beta} z^{2\gamma}, x^{4\alpha}, y^{4\beta},z^{4\gamma} \rangle.
  \]
  Consequently, $h$ is of the form
  $x^{\pm 2\alpha} y^{\pm 2\beta} z^{\pm 2 \gamma}$, and 
  $h \in \{ g^{\pm 1}, (g^{\pm 1})^x, (g^{\pm 1})^y, (g^{\pm 1})^{xy}
  \}$.

 \medskip

 \noindent \emph{Case 2}: $g \not \in A$.  Since $\mathrm{Aut}(G)$
 acts transitively on the non-trivial cosets of~$A$, we may assume
 without loss of generality that $g \in x A$, that is
 $g = x.x^{2\alpha} y^{2\beta} z^{2\gamma}$ with
 $\alpha, \beta, \gamma \in \mathbb{Z}$.  Clearly,
 $\langle h \rangle A = \langle g \rangle A$ implies
 $h = x.x^{2\alpha'} y^{2\beta'} z^{2\gamma'}$ for
 $\alpha', \beta', \gamma' \in \mathbb{Z}$.

 First consider the quotient
 $Q = G / \langle y^2, z^2 \rangle \cong D_\infty$, an infinite
 dihedral group.  Indeed, denoting the images of $x,y,\ldots$ in $Q$
 by $\tilde x, \tilde y, \ldots$, we observe that
 $Q = \langle \tilde y \rangle \ltimes \langle \tilde x \rangle$,
 where $\tilde y^2 = 1$, $\tilde x$ has infinite order and
 $\tilde x^{\tilde y} = \tilde x^{-1}$.  Note that
 $\langle \tilde h \rangle^Q = \langle \tilde g \rangle^Q = \langle
 \tilde x^{1+2\alpha} \rangle$
 implies $1+2\alpha' = \pm (1+2\alpha)$, hence $\alpha' = \alpha$ or
 $\alpha' = -1-\alpha$.  Replacing $h$ by $h^{-1}$, if necessary,
 we may suppose that $\alpha' = \alpha$.

  Conjugating $g$ by $a =  y^{2\lambda} z^{2\mu} \in A$, we obtain
  \[
  g^a = g (a^{-1})^g a = x.\big( x^{2\alpha} y^{2\beta} z^{2\gamma}
  \cdot (y^{-2\lambda} z^{-2\mu})^x \cdot
  y^{2\lambda} z^{2\mu} \big) 
  = x.(x^{2\alpha} y^{2\beta + 4\lambda} z^{2\gamma + 4\mu});
  \]
  we may therefore assume that $\beta, \gamma  \in \{0,1\}$ and,
  likewise, $\beta', \gamma' \in \{0,1\}$.

  Suppose that $(\beta,\gamma) \neq (\beta',\gamma')$.  The cyclic
  group $\langle g \rangle^G [G,G]/[G,G] \cong C_4$ contains
  \[
  g^2 \equiv (x^{1+2(\alpha+\gamma)} y^{2(\beta+\gamma)})^2 \equiv x^2
  \quad \text{and} \quad h^{-1}g \equiv x^{2(\gamma-\gamma')}
  y^{2(\beta-\beta'+\gamma-\gamma')} \quad \text{modulo~$[G,G]$.}
  \] 
 This implies that $\beta \neq \beta'$ and $\gamma
  \neq \gamma'$.  Swapping $g$ and $h$, if necessary, we deduce that
  \[
  g = x. x^{2\alpha},  \,  h = x. (x^{2\alpha} y^2 z^2) \qquad \text{or}
  \qquad g = x.(x^{2\alpha} z^2), \, h = x.(x^{2\alpha} y^2),
  \] 
  and a short calculation shows that $g^y = h^{-1}$.  For instance, in the first case 
  \[
  g^y = x^{-1}  y^2 z^2 x^{-2\alpha} = \big( x.(x^{2\alpha} y^2 z^2)
  \big)^{-1} = h^{-1}. \qedhere
  \] 
\end{proof}


\subsection{} The group defined in \eqref{equ:Passman-group}
allows for various generalisations; e.g., see~\cite{HiSa86,GuSi99}.
We are interested in one such variant, a family of centreless
torsion-free groups $G_p$ with $G_p \cong (C_p \times
C_p).\mathbb{Z}^{p^2-1}$, where $p$ denotes an arbitrary
prime.  In the end, we will 
need $H = G_3 \cong (C_3 \times C_3).\mathbb{Z}^8$, which we introduce
in advance by the following presentation:
\begin{equation} \label{equ:generalised-Passman} 
  \begin{split}
    H =& \langle u, v, e_\infty, \hat e_\infty, e_0, \hat e_0, e_1,
    \hat e_1, e_2, \hat e_2\mid\\
    &\qquad[a,b] \quad\text{for } a, b\in \{e_\infty, \hat e_\infty,
    e_0, \hat e_0, e_1, \hat e_1, e_2, \hat e_2\}, \\
    &\qquad a^z=c_{a,z}\quad\text{for } a\in
    \{e_\infty, \hat e_\infty, e_0, \hat e_0, e_1, \hat e_1, e_2,
    \hat e_2\}, z\in\{u,v\}, \\
   &\qquad u^3=e_\infty^{\, -2} \hat e_\infty^{\, -1}, \quad v^3 =
   e_0^{\, 2} \hat e_0, \quad  [u,v] = e_\infty e_0 e_1 e_2
    \rangle
  \end{split}
\end{equation}
where the words $c_{a,z}$ are given as in the following table:
\[
{\setlength\arraycolsep{10pt}
\begin{array}{c|cccccccc}
  &\, e_\infty\, & \hat e_\infty & e_0 & \hat e_0 & e_1 & \hat e_1 &
                                                                     e_2 & \hat e_2 \\[1pt]
  \hline \\[-12pt]
  u & e_\infty & \hat e_\infty & \hat e_0  &  e_0^{\,-1} \hat e_0^{\,-1} & \hat e_1
                                                        & e_1^{\,-1} \hat e_1^{\,-1} & \hat e_2 &
                                                                                                  e_2^{\,-1}
                                                                                                  \hat
                                                                                                  e_2^{\,-1}\\ 
  v & \hat e_\infty & e_\infty^{\,-1} \hat e_\infty^{\,-1} & e_0 &
                                                                   \hat
                                                                   e_0
                                                  & \hat e_1 &
                                                               e_1^{\,-1} \hat e_1^{\,-1} 
                                                                   &\,
                                                                     e_2^{\,-1}
                                                                     \hat e_2^{\,-1} & e_2 
\end{array}}
\]

By the first set of relations,
$B = \langle e_\infty, \hat e_\infty, e_0, \hat e_0, e_1, \hat e_1,
e_2, \hat e_2 \rangle$
is an abelian subgroup of~$H$.  The second set of relations shows that
$B \trianglelefteq H$, and from the third set of relations, we see
that $H/B = \langle \ol u,\ol v \rangle \cong C_3\times C_3$.

In order to motivate the presentation~\eqref{equ:generalised-Passman}
and to justify further structural properties of~$H$, we now construct,
more generally, the group $G = G_p$ for an arbitrary prime~$p$.
Consider an elementary abelian group
$\langle \ol u, \ol v \rangle \cong C_p \times C_p$, and fix a complex
primitive $p$th root of unity $\zeta$, satisfying the equation
$\sum_{i=0}^{p-1} \zeta^i = 0$.  We use the ring of integers
$\mathbb{Z}[\zeta] = \mathbb{Z} + \mathbb{Z} \zeta + \ldots +
\mathbb{Z} \zeta^{p-2}$
to describe explicitly a collection of $(p-1)$-dimensional
indecomposable $\mathbb{Z} \langle \ol u, \ol v \rangle$-lattices,
where the action of $\langle \ol u, \ol v \rangle$ factors through a
group of order~$p$; cf.\ \cite[(34.31)]{CuRe81}.  Concretely, we
consider $p+1$ such lattices $M_\infty, M_0, M_1, \ldots, M_{p-1}$,
each modeled on $\mathbb{Z}[\zeta]$ with the actions of $\ol u$ and
$\ol v$ specified by the columns of the following table.

\begin{center}
  \begin{tabular}{c| p{1.5cm} p{1.5cm} p{1.5cm} p{1.5cm}  p{1.5cm} p{1.5cm}}
    & $M_\infty$ & $M_0$ & $M_1$ & $M_2$ & $\cdots$ & $M_{p-1}$ \\
    \hline\hline
    \multicolumn{7}{c}{action on $M_i = \mathbb{Z}[\zeta]$ via
    multiplication by} \\ 
    \hline
    $\ol u$ & $1$ & $\zeta$ & $\zeta$ & $\zeta$ & $\cdots$ & $\zeta$ \\
    $\ol v$ & $\zeta$ & $1$ & $\zeta$ & $\zeta^2 $ & $\cdots$ & $\zeta^{p-1}$
  \end{tabular}
\end{center}

Here the index set $I = \{\infty\} \cup \{ i \mid 0 \le i \le p-1 \}$
parameterises naturally the various kernels of the actions of
$\langle \ol u, \ol v \rangle$ and can be thought of as the projective
line over $\mathbb{Z}/p\mathbb{Z}$.  We build $G$ as a non-split
extension $\langle \ol u, \ol v \rangle.B$, where
$B = B_p = \bigoplus_{i \in I} M_i \trianglelefteq G$ is isomorphic to
$\mathbb{Z}^{p^2-1}$ as an additive group.  For $i \in I$, we denote
by $e_i$ the generator $1$ of $M_i$.  In the special case $p=3$, the
additional generators $\hat e_i$ in the
presentation~\eqref{equ:generalised-Passman} correspond to the
elements $\zeta \in M_i$, for $i \in \{\infty,0,1,2\}$, and the
relations of the form $a^z=c_{a,z}$ express the action of $u,v$ on $B$
via~$\ol u, \ol v$.  Using the module-based notation and writing
$\pi = \zeta -1$, the general definition of $G$ is completed by
specifying for the generators $u,v$ and $e_i$, $i \in I$, the
additional relations:
\[
u^p = p \pi^{-1} e_\infty, \qquad  v^p  = -p \pi^{-1} e_0 \qquad
\text{and} \qquad
[u,v] = \textstyle\sum_{i \in I} e_i.
\] 

For concreteness, it is readily verified that these relations and the actions
described above hold for the following matrices, thus yielding a
faithful matrix representation of $G$
in $\mathrm{GL}_{p+2}(\mathbb{Q}(\zeta))$:
\[
U =
\left[\begin{smallmatrix}
1 & \nicefrac{1}{\pi}& 0 & 0 & \cdots & 0 \\ 
 & 1 & & & \\
 & & \zeta & & \\
 & & & \zeta & & \\
 & & & & \cdots & \\
 & & & & & \zeta
\end{smallmatrix}\right],
\quad V =
\left[\begin{smallmatrix}
  1 & 0 & \nicefrac{-1}{\pi} & \nicefrac{-1}{\pi}
       & \cdots & \nicefrac{-1}{\pi} \\
  & \zeta & & & \\
  & & 1 & & \\
  & & & \zeta & & \\
  & & & & \cdots & \\
  & & & & & \zeta^{p-1}
\end{smallmatrix}\right]
\]
and
\[
E_\infty =
\left[\begin{smallmatrix}
1 & 1 & 0 & 0 & \cdots & 0 \\ 
 & 1 & & & \\
 & & 1 & & \\
 & & & 1 & & \\
 & & & & \cdots & \\
 & & & & & 1
\end{smallmatrix}\right],
\quad
E_0 =
\left[\begin{smallmatrix}
1 & 0 & 1 & 0 & \cdots & 0 \\ 
 & 1 & & & \\
 & & 1 & & \\
 & & & 1 & & \\
 & & & & \cdots & \\
 & & & & & 1
\end{smallmatrix}\right],
\quad
\ldots,
\quad
E_{p-1} =
\left[\begin{smallmatrix}
1 & 0 & 0 & \cdots & 0 & 1 \\ 
 & 1 & & & \\
 & & 1 & & \\
 & & & \cdots & & \\
 & & & & 1 & \\
 & & & & & 1
\end{smallmatrix}\right].
\]
It is easily checked that $G$ is centreless and torsion-free.
Furthermore, $B$ is the unique self-centralising finite-index
subgroup, and so, in particular, characteristic in~$G$.  

Subject to the usual conventions $\infty + 1 = \infty$,
$\infty^{-1} = 0$ etc., the operations $i \mapsto i+1$ and
$i \mapsto 1/i$ modulo $p$ induce bijections $I \to I$.  Representing
the elements of $B$ as linear combinations
$\sum_{i\in I} b_i(\zeta) \, e_i$, with polynomials
$b_i \in \mathbb Z[t]$ for $i \in I$, we can specify
two automorphisms $\sigma,\tau \in \mathrm{Aut}(G)$ by
\begin{align*}
  u \sigma = u, \quad v \sigma = uv, \quad
    & \left( {\textstyle\sum_{i\in I} b_i(\zeta) \, e_i} \right) \sigma
    = {\textstyle\sum_{i\in I} b_{i+1}(\zeta) \,  e_i}, \\ 
   u \tau =v, \quad v \tau = u,\phantom{v} \quad
    & \left( {\textstyle\sum_{i\in I} b_i(\zeta) \, e_i} \right ) \tau =
    -b_0(\zeta) \,e_{\infty} - b_\infty(\zeta) \, e_0 -
      {\textstyle\sum_{i=1}^{p-1} b_{1/i}(\zeta^i) \, e_i}.
\end{align*}
The induced automorphisms on
$G/B = \langle\ol u,\ol v\rangle\cong C_p\times C_p$ correspond to
$\left(\begin{smallmatrix} 1 & 0 \\ 1 & 1\end{smallmatrix}\right),
\left(\begin{smallmatrix} 0 & 1 \\ 1 & 0 \end{smallmatrix}\right) \in
\mathrm{GL}(2,p)$.
In particular, $\mathrm{Aut}(G)$ permutes transitively the non-trivial
cosets $gB$ and operates $2$-transitively on $\{ M_i \mid i \in I \}$.

We remark in passing that, for $p>2$, the group $G = G_p$ is not
$2$-generated and thus not isomorphic to the universal group $K(p)$
studied in~\cite{GuSi99}.  Indeed, for $p>2$ we observe that
$G/[G,G] \cong C_p^{\, p+2}$, because 
\[
[G,G] = \left\{ {\textstyle\sum_{i \in I} b_i(\zeta) \, e_i} \in B \mid b_\infty(1)
  = b_0(1) = \ldots = b_{p-1}(1) \right\}.
\]

Recall that $G_2$ has the Magnus property by
Proposition~\ref{pro:G-group}.  In contrast, for $p>3$ the group
$G = G_p$ does not have the Magnus property.  Indeed, for any
$\alpha \in \mathbb Z[\zeta]$ the normal closure of
$\alpha e_0 \in M_0 \subseteq G$ in $G$ is simply the
$\mathbb Z\langle\ol u,\ol v\rangle$-submodule of $M_0$ generated
by~$\alpha e_0$.  For $p>3$, the group of units
$\mathbb{Z}[\zeta]^\times$ has torsion-free rank $(p-3)/2 \ge 1$ by
Dirichlet's theorem, so there are infinitely many elements
$\epsilon e_0$, $\epsilon \in \mathbb{Z}[\zeta]^\times$, with normal closure
$\langle \epsilon e_0 \rangle^G = M_0$.  However, each
element of $B$ has at most $p^2$ conjugates in~$G$, so these cannot
all be conjugate or inverse-conjugate to one another.

Therefore we restrict our focus to the case $p=3$ and consider the
group $H = G_3$.  As customary, we denote by $\omega$, in place of
$\zeta$, a complex primitive third root of unity.  The ring of
Eisenstein integers $\mathbb{Z}[\omega]$ has the finite group of units
$\mathbb{Z}[\omega]^\times = \{ \pm 1, \pm \omega, \pm \omega^2 \}$.
We continue to write $\pi = \omega-1$ and note that now
$\pi^2 \mathbb{Z}[\omega] = 3 \mathbb{Z}[\omega]$.

\begin{prop} \label{pro:H-group} The group $H$ defined in
  \eqref{equ:generalised-Passman} has the Magnus property.
\end{prop}

\begin{proof}
  It is convenient to use the module-based notation for $H = G_3$,
  explained above.  We fix $g \in H$ and show that any $h \in H$ with
  $\langle g \rangle^H = \langle h \rangle^H$ is conjugate to $g$ or
  to~$g^{-1}$.  We distinguish two cases.

  \medskip

  \noindent \emph{Case 1}: $g \in B$.  In this case
  $g = \alpha e_\infty + \beta e_0 + \gamma e_1 + \delta e_2$ with
  $\alpha, \beta, \gamma, \delta \in \mathbb{Z}[\omega]$.  The normal
  closure $\langle g \rangle^H$ is equal to the
  $\mathbb{Z} \langle \ol u, \ol v \rangle$-submodule of
  $B = \bigoplus_{i\in I} M_i$ generated by~$g$.  For simplicity, we
  drop the $e_i$ from the notation and write
  $g = (\alpha,\beta,\gamma,\delta)$.  For any $i \in I$,
  multiplication by any non-zero element of $\mathbb{Z}[\omega]$
  yields an injective $\mathbb{Z} \langle \ol u, \ol v \rangle$-module
  endomorphism of~$M_i$.  `Scaling' each coordinate, we may thus
  suppose without loss of generality that
  $\alpha,\beta,\gamma,\delta \in \{0,1\}$.

  If $g = (0,0,0,0)$ there is nothing to show.  Suppose next that $g$
  has precisely one non-zero coordinate, $g = (1,0,0,0)$ say.
  Clearly, the $\mathbb{Z} \langle \ol u, \ol v \rangle$-module
  generated by $g$ is equal to the outer direct sum
  $\mathbb{Z}[\omega] \oplus \{0\} \oplus \{0\} \oplus \{0\}$.  The
  generators of this module are of the form $(\epsilon,0,0,0)$ with
  $\epsilon \in \{ \pm 1, \pm \omega, \pm \omega^2 \}$.  Clearly, each
  of these is equal to $\pm {\ol v}^{\, j} (1,0,0,0)$, that is
  $(g^{\pm 1})^{v^j}$, for suitable $j \in \{0,1,2\}$.

  Now suppose that $g$ has precisely two non-zero coordinates,
  $g = (1,1,0,0)$ say.  Then the
  $\mathbb{Z} \langle \ol u, \ol v \rangle$-module generated by $g$ is
  equal to
  \[
  \{ (0,0,0,0), (1,1,0,0), (2,2,0,0) \} + \big( \pi \mathbb{Z}[\omega]
  \oplus \pi \mathbb{Z}[\omega] \oplus \{0\} \oplus \{0\} \big).
  \]
  The generators of this module are of the form
  $(\epsilon_1,\epsilon_2,0,0)$ or $(-\epsilon_1,-\epsilon_2,0,0)$
  with $\epsilon_1, \epsilon_2 \in \{ 1, \omega, \omega^2 \}$.
  Clearly, each of these is equal to
  $\pm {\ol u}^{\, i} {\ol v}^{\, j} (1,1,0,0)$, that is
  $(g^{\pm 1})^{u^i v^j}$, for suitable
  $i,j \in \{0,1,2\}$.

  Now suppose that $g$ has precisely three non-zero coordinates,
  $g = (1,1,1,0)$ say.  Then the
  $\mathbb{Z} \langle \ol u, \ol v \rangle$-module generated by $g$ is
  equal to
  \[
  \mathbb{Z} (1,1,1,0) + \mathbb{Z} (0,\pi,\pi,0) + \mathbb{Z}
  (\pi,0,\pi,0) + \big( 3 \mathbb{Z}[\omega] \oplus 3
  \mathbb{Z}[\omega] \oplus 3 \mathbb{Z}[\omega] \oplus \{0\} \big).
  \]
  The generators of this module are of the form
  $(\epsilon_1,\epsilon_2,\epsilon_3,0)$ or
  $(-\epsilon_1,-\epsilon_2,-\epsilon_3,0)$ with
  $\epsilon_1, \epsilon_2, \epsilon_3 \in \{ 1, \omega, \omega^2 \}$,
  where $\epsilon_3$ is uniquely determined by
  $(\epsilon_1,\epsilon_2)$.  Again, each of these is equal to
  $\pm {\ol u}^{\, i} {\ol v}^{\, j} (1,1,1,0)$, that is
  $(g^{\pm 1})^{u^i v^j}$, for suitable $i,j \in \{0,1,2\}$.

  Finally, let $g=(1,1,1,1)$.  Then the
  $\mathbb{Z} \langle \ol u, \ol v \rangle$-module generated by $g$ is equal to
  \[
  \mathbb{Z} (1,1,1,1) + \mathbb{Z} (0,\pi,\pi,\pi) + \mathbb{Z}
  (\pi,0,\pi,-\pi) + \big( 3 \mathbb{Z}[\omega] \oplus 3
  \mathbb{Z}[\omega] \oplus 3 \mathbb{Z}[\omega] \oplus 3
  \mathbb{Z}[\omega] \big).
  \]
  The generators of this module are of the form
  $(\epsilon_1,\epsilon_2,\epsilon_3,\epsilon_4)$ or
  $(-\epsilon_1,-\epsilon_2,-\epsilon_3,-\epsilon_4)$ with
  $\epsilon_1, \epsilon_2, \epsilon_3, \epsilon_4 \in \{ 1, \omega,
  \omega^2 \}$,
  where $(\epsilon_3,\epsilon_4)$ is uniquely determined by
  $(\epsilon_1,\epsilon_2)$.  Again, each of these is equal to
  $\pm {\ol u}^{\, i} {\ol v}^{\, j} (1,1,1,1)$, that is
  $(g^{\pm 1})^{u^i v^j}$, for suitable
  $i,j \in \{0,1,2\}$.

 \medskip

 \noindent \emph{Case 2}: $g \not \in B$.  Since $\mathrm{Aut}(H)$
 acts transitively on the non-trivial cosets of~$B$, we may assume
 without loss of generality that $g \in u B$, that is
 $g = u.(\alpha,\beta,\gamma,\delta)$ with
 $\alpha, \beta, \gamma, \delta \in \mathbb{Z}[\omega]$, where we
 suppress $e_1,e_2,e_3,e_4$ from the notation as before.  Suppose that
 $h \in H$ has the same normal closure in $H$ as~$g$.  Then
 $\langle h \rangle B = \langle g \rangle B = \langle u \rangle B$
 shows that, replacing $h$ by $h^{-1}$ if necessary, we may
 suppose that $h = u.(\alpha',\beta',\gamma',\delta')$ for
 $\alpha', \beta', \gamma', \delta' \in \mathbb{Z}[\omega]$.

 First consider the quotient $Q = H / (M_0 \oplus M_1 \oplus M_2)$.
 We denote the image of $z \in H$ in $Q$ by $\tilde{z}$.  Observe that
 $\tilde{v}$ has order $3$ and that
 $Q = \langle \tilde{v} \rangle \ltimes \langle \tilde{u},
 [\tilde{u},\tilde{v}] \rangle$
 is naturally isomorphic to
 $\langle w \rangle \ltimes \pi^{-1} \mathbb{Z}[\omega]$, where
 $w$ has order $3$ and acts via multiplication by $\omega$ on the
 additive group $\pi^{-1} \mathbb{Z}[\omega]$, by virtue
 of $\tilde{v} \mapsto w$ and
 $\tilde{u} \mapsto \pi^{-1}$.  Note that the normal
 closure of $\tilde{g}$ corresponds to
 $(\pi^{-1} + \alpha) \mathbb{Z}[\omega]$.  Likewise the
 normal closure of $\tilde{h}$ corresponds to
 $(\pi^{-1} + \alpha') \mathbb{Z}[\omega]$.  This
 implies that
  \[
  \pi^{-1} + \alpha' = \epsilon \big( \pi^{-1} + \alpha \big) \qquad
  \text{for a suitable
    $\epsilon \in \{\pm 1, \pm \omega, \pm \omega^2\}$.}
  \]
  
  Since $\alpha, \alpha' \in \mathbb{Z}[\omega]$, we deduce that
  $\epsilon \in \{1,\omega,\omega^2\}$ so that there are at most three
  possible values for $\alpha'$, depending on $\alpha$.  Conjugating
  $h$ by $v^j$ for a suitable $j \in \{0,1,2\}$, we may thus assume
  that $\alpha' = \alpha$.

  Conjugating $g$ by $b = (0,\lambda,\mu,\nu) \in B$, we obtain
  \begin{multline*}
    g^b = g (b^{-1})^g b = u.\big( (\alpha, \beta, \gamma, \delta) +
    \ol u (0,-\lambda,-\mu,-\nu) +
    (0,\lambda,\mu,\nu) \big) ) \\
    = u.(\alpha, \beta - \pi \lambda, \gamma - \pi \mu, \delta - \pi \nu).
  \end{multline*}
  We may therefore assume that $\beta, \gamma, \delta \in \{0,1,2\}$ and,
  likewise, $\beta', \gamma', \delta' \in \{0,1,2\}$.

  For a contradiction, suppose that
  $(\beta,\gamma,\delta) \neq (\beta',\gamma',\delta')$.  Then the
  normal closure of $g$ in $H$ would contain
  \[
  g = u.(\alpha,\beta,\gamma,\delta) \qquad \text{and} \qquad h^{-1}g =
  (0,\beta-\beta',\gamma-\gamma',\delta-\delta').
  \] 
  But this would imply that $\langle g \rangle^H [H,H]/[H,H]$ was not cyclic, a
  contradiction.
\end{proof}


\begin{prop} \label{pro:G-H-product}
  The direct product $G \times H$ of the crystallographic groups $G$
  and $H$ defined in \eqref{equ:Passman-group} and \eqref{equ:generalised-Passman} 
  does not have the Magnus property.
\end{prop}

\begin{proof}
  We consider $g = x \in G$ and $h = e_\infty = (1,0,0,0) \in B \leq H$.
  Since $g$ has order $4$ modulo $[G,G]$, we see that
  $g \not \equiv g^{-1}$ modulo $[G,G]$.  Consequently, $g$ and
  $g^{-1}$ are not conjugate in~$G$.  Similarly, $h$ generates a
  cyclic group of order $3$ modulo $[H,H]$, and thus $h$ and $h^{-1}$
  are not conjugate in~$H$.

  Observe that
  \[
  M = \langle [g,g'] \mid g' \in G \rangle^G = [G,G]
  \]
  is normal and of index $4$ in $\langle g \rangle^G$.  Suppressing
  the $e_i$ as before, we see that
  \[
  N = \langle [h,h'] \mid h' \in H \rangle^H =
  \pi \mathbb{Z}[\omega] \oplus \{0\} \oplus \{0\} \oplus
  \{0\}
  \]
  has index $3$ in
  $\langle h \rangle^H = \mathbb{Z}[\omega] \oplus \{0\} \oplus \{0\} \oplus
  \{0\}$.

  This implies that each of the groups $\langle (g,h) \rangle^{G \times H}$ and
  $\langle (g,h^{-1}) \rangle^{G \times H}$ contains $M \times N$ and,
  modulo $M \times N$, forms a subdirect product of
  \[
  (\langle g \rangle^G \times \langle h \rangle^H) / (M \times N) \cong
  (\langle g \rangle^G/M) \times (\langle h \rangle^H/N) \cong C_4
  \times C_3.
 \] 
 We deduce that
 \[
 \langle (g,h) \rangle^{G \times H} = \langle g \rangle^G \times
 \langle h \rangle^H = \langle (g,h^{-1}) \rangle^{G
    \times H}.
 \]
  According to Lemma~\ref{lem:i-ii-iii}, the group $G \times H$ does
  not have the Magnus property.
\end{proof}

Propositions~\ref{pro:G-group}, \ref{pro:H-group} and \ref{pro:G-H-product}
clearly imply Proposition~\ref{pro:not-preseverved-torsion-free} in
the introduction.

\medskip

\begin{acknow}
 We are grateful to Steffen Kionke for drawing our attention to the
 work of Hiller and Sah~\cite{HiSa86}.  
\end{acknow}



\end{document}